\documentclass[letterpaper,11pt]{article}
\usepackage[margin=1in]{geometry}  % set the margins to 1in on all sides

\usepackage{listings}
\usepackage{amsmath}
\usepackage{amsthm}
\usepackage{tikz}
\usepackage{caption}
\usepackage{array}
\usepackage{mdwmath}
\usepackage{multirow}
\usepackage{mdwtab}
\usepackage{eqparbox}
\usepackage{multicol}
\usepackage{amsfonts}
\usepackage{tikz}
\usepackage{multirow,bigstrut,threeparttable}
\usepackage{amsthm}
\usepackage{array}
\usepackage{bbm}
\usepackage{subfigure}
\usepackage{epstopdf}
\usepackage{mdwmath}
\usepackage{mdwtab}
\usepackage{eqparbox}
\usepackage{tikz}
\usepackage{latexsym}
\usepackage{cite}
\usepackage{amssymb}
\usepackage{bm}
\usepackage{amssymb}
\usepackage{graphicx}
\usepackage{mathrsfs}
\usepackage{epsfig}
\usepackage{psfrag}
\usepackage{setspace}
\usepackage[%dvips,
            CJKbookmarks=true,
            bookmarksnumbered=true,
            bookmarksopen=true,
            colorlinks=true,
            citecolor=red,
            linkcolor=blue,
            anchorcolor=red,
            urlcolor=blue
            ]{hyperref}
\usepackage{algorithm}
\usepackage{algpseudocode}
\usepackage{stfloats}
\theoremstyle{plain}
\newtheorem{theorem}{Theorem}

\newtheorem{lemma}{Lemma}

\newtheorem{corollary}{Corollary}
\newtheorem{definition}{Definition}
\newtheorem{conjecture}{Conjecture}

\newtheorem*{question*}{Question}

\def \bP {\mathbb{P}}
\def \bE {\mathbb{E}}
\def \bR {\mathbb{R}}

\usepackage{xspace}

\newcommand{\naturals}{\mathbb{N}}

\definecolor{myblue}{rgb}{.8, .8, 1}
\definecolor{mathblue}{rgb}{0.2472, 0.24, 0.6} % mathematica's Color[1, 1--3]
\definecolor{mathred}{rgb}{0.6, 0.24, 0.442893}
\definecolor{mathyellow}{rgb}{0.6, 0.547014, 0.24}

\usepackage{cleveref}
\crefname{lemma}{Lemma}{Lemmas}
\Crefname{lemma}{Lemma}{Lemmas}
\crefname{thm}{Theorem}{Theorems}
\Crefname{thm}{Theorem}{Theorems}
\crefname{assumption}{Assumption}{Assumptions}
\Crefname{assumption}{Assumption}{Assumptions}
\crefformat{equation}{(#2#1#3)}

\usepackage{enumerate}
\begin{document}
	
%	\begin{frontmatter}
		
		% "Title of the paper"
		\title{Expectation of the Largest bet size in Labouchere System}
		\author{Yanjun Han and Guanyang Wang\thanks{Yanjun Han is with the Department of Electrical Engineering, Stanford University. Guanyang Wang is with the Department of Mathematics, Stanford University. }}
%		\runtitle{Labouchere System}
		
		% indicate corresponding author with \corref{}
		% \author{\fnms{John} \snm{Smith}\corref{}\ead[label=e1]{smith@foo.com}\thanksref{t1}}
		% \thankstext{t1}{Thanks to somebody} 
		% \address{line 1\\ line 2\\ printead{e1}}
		% \affiliation{Some University}
		
%		\begin{aug}
%			% \author{\fnms{Persi} \snm{Diaconis}\thanksref{}\ead[label=e1]{diaconis@math.stanford.edu}},
%			\author{\fnms{Yanjun} \snm{Han}\thanksref{}\ead[label=e2]{yjhan@stanford.edu}}
%			\and
%			\author{\fnms{Guanyang} \snm{Wang}\thanksref{}\ead[label=e3]{guanyang@stanford.edu}}
%			\runauthor{Han and Wang}
%			\affiliation{Stanford University}
%			
%			%	\address{Persi Diaconis\\
%			%		Department of Mathematics and Statistics \\
%			%		Stanford University \\
%			%		Stanford, CA, 94305\\
%			%		USA\\
%			%		\printead{e1}\\}
%			
%			\address{Yanjun Han \\
%				Department of Electrical Enginnering \\
%				Stanford University \\
%				Stanford, CA, 94305\\
%				USA\\
%				\printead{e2}\\}
%			
%			\address{Guanyang Wang \\
%				Department of Mathematics \\
%				Stanford University \\
%				Stanford, CA, 94305\\
%				USA\\
%				\printead{e3}\\}
%		\end{aug}

		%\begin{abstract}
		%\end{abstract}
		
%		\begin{keyword}[class=MSC]
%			\kwd[Primary ]{60G40}
%			\kwd[; secondary ]{60C05}
%		\end{keyword}
%		
%		\begin{keyword}
%			\kwd{Labouchere system; Gambling theory; Martingale; Combinatorics}
%		\end{keyword}
			\maketitle
		\begin{abstract}
			For Labouchere system with winning probability $p$ at each coup, we prove that the expectation of the largest bet size under any initial list is finite if $p>\frac{1}{2}$, and is infinite if $p\le \frac{1}{2}$, solving the open conjecture in \cite{grimmett2001one}. The same result holds for a general family of betting systems, and the proof builds upon a recursive representation of the optimal betting system in the larger family.
		\end{abstract}
		%
		%\begin{keyword}[class=AMS]
		%\kwd[Primary ]{62G05}
		%\kwd[; secondary ]{62C20}
		%\end{keyword}
		%\begin{keyword}
		%\kwd{Nonparametrics; Non-smooth functional estimation; Minimax rate; Polynomial approximation.}
		%\end{keyword}
%	\end{frontmatter}
	
	\tableofcontents 
	\section{Introduction}
	The Labouchere system, also known as the cancellation system, is one of the most well-known betting systems used in roulette. It was popularized by Henry Du Pr\'e Labouchere, an English politician, writer and journalist. Before the betting, the bettor chooses an initial list $L_0$ of positive real numbers (e.g., $L_0=(1,2,3,4)$). During each bet, the bet size equals the sum of the first and last numbers on the list (if only one number remains on the list, then the bet size equals that number). After a win, the first and last terms are canceled from the list; after a loss, the amount just lost is appended to the last term of the list. This system is continued until the list is empty. Table \ref{tab.Illustration of Labouchere system} illustrates an example of Labouchere system.
	
	\begin{table}[htbp]
		\centering
		\caption{An illustration of the Labouchere system with initial list $L_0=(1, 2, 3, 4)$.}
		\label{tab.Illustration of Labouchere system}
		\begin{tabular}{|c|c|c|c|c|}
			\hline
			Coup $n$ & Bet Size $B_n$ & Result & List $L_n$   & 
			\begin{tabular}[c]{@{}c@{}}Target Profit $T_n$\end{tabular} \\ \hline
			&     &        & 1, 2, 3, 4    &             10                                                \\ \hline
			1     & 5   & Win    & 2, 3          & 5                                                           \\ \hline
			2     & 5   & Loss   & 2, 3, 5       &  10                                                           \\ \hline
			3     & 7   & Loss   & 2, 3, 5, 7    &   17                                                          \\ \hline
			4     & 9   & Loss   & 2, 3, 5, 7, 9 &  26                                                         \\ \hline
			5     & 11  & Win    & 3, 5, 7       &  15                                                          \\ \hline
			6     & 10  & Loss   & 3, 5, 7, 10   & 25                                                         \\ \hline
			7     & 13  & Win    & 5, 7          &  12                                                          \\ \hline
			8     & 12  & Win    & $\varnothing$   & 0                                                          \\ \hline
		\end{tabular}
	\end{table}
	
	We introduce the following notations.  Let $L_n$ be the list after the $n$-th coup, $l_n$ be the corresponding list length, $B_n$ be the bet size at $n$-th coup, $T_n$ be the remaining target profit (i.e., the sum of the numbers in the list) after $n$-th coup, and $N$ be the stopping time that the list first becomes empty, i.e., $L_N = \varnothing$. In this paper, we investigate the behavior of the largest bet size $B^\star \triangleq \max_{1\le n\le N} B_n$ (or $\sup_{n\ge 1}B_n$ if $N=\infty$) in the Labouchere system, and in particular, whether or not $B^\star$ has a finite expectation.
	
	There is very limited literature on analyzing the Labouchere system. Let $p\in [0,1]$ be the winning probability at each coup, where we assume that the outcomes at different coups are independent. By the standard theory of asymmetric random walk, it is straightforward to show that $N < \infty$ almost surely if and only if $p \geq \frac 13$ and $\bE[N]<\infty$ if and only if $p>\frac{1}{3}$. Downton \cite{downton1980note} found a recursion for the distribution of the stopping time $N$ in the case that the initial list $L_0$ is $(1, 2, 3, 4)$, and Ethier \cite{ethier2008absorption} generalized this result to arbitrary initial list and gave an explicit formula using a generalized version of the ballot theorem \cite{bertrand1887solution,barbier1887generalisation}. Specifically, the stopping time $N$ has finite $k$-th moment for any $k$ if and only if $p>\frac{1}{3}$. However, Grimmett and Stirzaker \cite[Problem 12.9.15]{grimmett2001one} showed that both $\max_{1\le n\le N} T_n$ and $\sum_{n=1}^N B_n$ have infinite expectations if $p=\frac{1}{2}$. It was also stated in \cite{grimmett2001one} that $\bE[B^\star]=\infty$, but we were informed by Ethier that the proof was incomplete (via an email exchange between him and Grimmett in February 2006). Hence, it remains an open conjecture for more than a decade if the largest bet size $B^\star$ also has an infinite expectation when $\frac{1}{3}<p\le \frac{1}{2}$, which is the main focus of this paper. 
	
	There is also another betting system which is similar to the Labouchere system, i.e., the Fibonacci system. Instead of considering the first and last numbers in the list at each coup, the \emph{last} two numbers are added or canceled in the Fibonacci system. Ethier \cite{ethier2010doctrine} showed that $\bE B^\star  = +\infty$ in Fibonacci system if and only if $p\le \frac{1}{2}$. However, the proof heavily relies on the fact that any list in a Fibonacci system is uniquely determined by its length, which does not hold for the Labouchere system where the list evolves in a more complicated ``history dependent" manner.

	%This report is organized as follows. In Section \ref{related works}, we briefly review the present results about Labouchere system and other similar systems. In Section \ref{sec.list_system}, we introduce a large class of list systems and show that $\bE[B^\star]$ is finite if $p>\frac{1}{2}$ and is infinite if $p<\frac{1}{2}$ under general list systems. The Labouchere system is analyzed in Section \ref{sec.Labouchere_system}, where we compute the exact geometric growth of $B^\star$ conditioning on $N$, and prove that $\bE[B^\star]=\infty$ for $p=\frac{1}{2}$. The conclusions are drawn in Section \ref{sec.conclusion}. 
	
	\section{Main Results}
	To study the Labouchere system, we first introduce a larger family of betting systems called \emph{$(a,b)$-list systems}: 
	\begin{definition}[$(a,b)$-List System]\label{def.list_system}
		Let $a<0\le b$ be integers. An $(a,b)$-list system consists of a target sequence $\{T_n\}$, a bet sequence $\{B_n\}$ and a length sequence $\{l_n\}$, which evolve as follows:
		\begin{enumerate}
			\item At the beginning, $T_0>0$ and $l_0\in \{1,2,\cdots\}$; 
			\item At $n$-th coup, the system makes a bet size $B_n\in [0,T_{n-1}]$ which may depend on the entire history. Then the target and length sequences evolve as
			\begin{align*}
			T_n = \begin{cases}
			T_{n-1} - B_n & \text{if wins} \\
			T_{n-1} + B_n & \text{if loses}
			\end{cases}, \qquad l_n = \begin{cases}
			(l_{n-1} + a)_+ & \text{if wins} \\
			l_{n-1} + b & \text{if loses}
			\end{cases}. 
			\end{align*}
			\item Termination condition: let $N=\inf\{n: l_n=0\}$ be the stopping time that the length becomes zero, we must have $T_n=l_n=0$ for any $n\ge N$ and $B_n=0$ for any $n>N$. 
		\end{enumerate}
	\end{definition}
	In such a list system, target $T_n$ represents the remaining amount of money one would like to earn at the end of $n$-th coup; consequently, $T_n$ shrinks after a win, and increases after a loss. Length $l_n$ represents the length of the ``list'' at $n$-th coup, where it may be some real/virtual list which governs the betting process. For example, the well-known martingale system (where the bet is doubled after each loss) belongs to the $(-1,0)$-list system with $l_0=1$ and $B_n=T_{n-1}$, and both Labouchere and Fibonacci systems fall into the category of $(-2,1)$-list systems. The termination condition ensures that, as long as the list length $l_n$ hits zero, the target must be fulfilled as well (i.e., $T_n=0$), and the betting process terminates.
	
	In this paper, we only consider $(-2,1)$-list systems where the Labouchere system is included, but our results and proof techniques are generalizable to general $(a,b)$-list systems. Our first result characterizes the behavior of the largest bet size $B^\star$ under general list systems: 
	\begin{theorem}\label{thm.general_list_system}
		For any $(-2,1)$-list system, the following holds: 
		\begin{enumerate}
			\item If $p>\frac{1}{2}$, we have $\bE[B^\star]<\infty$; 
			\item If $\frac{1}{3} < p < \frac{1}{2}$, we have $\bE[B^\star]=\infty$; 
			\item If $p\le \frac{1}{3}$ and $B_n\ge c_1l_{n-1} + c_2$ for some constants $c_1>0, c_2\in \bR$ almost surely, we have $\bE[B^\star]=\infty$. 
		\end{enumerate}
	\end{theorem}
	
	Theorem \ref{thm.general_list_system} shows that for any $(-2,1)$-list systems, the expectation $\bE[B^\star]$ of the largest bet size $B^\star$ has a phase transition at $p=\frac{1}{2}$: the expectation is finite if the player is favored, and is infinite if the house takes the advantage. Consequently, we have the following corollary: 
	\begin{corollary}\label{cor.p_not_half}
		For the Labouchere system with any initial list, we have $\bE[B^\star]<\infty$ if $p>\frac{1}{2}$ and $\bE[B^\star]=\infty$ if $p<\frac{1}{2}$. 
	\end{corollary}
	
	The fair-game case $p=\frac{1}{2}$ requires more delicate analysis, and is summarized in the following theorem: 
	\begin{theorem}\label{thm.constrained_bet}
		Let $(\overline{b}_l)_{l=1}^\infty, (\underline{b}_l)_{l=1}^\infty$ be two sequences taking value in $[0,1]$. Suppose that a $(-2,1)$-list system satisfies $T_{n-1}\underline{b}_{l_{n-1}} \le B_n \le T_{n-1}\overline{b}_{l_{n-1}}$ for any $n$, and one of the following conditions holds: 
		\begin{enumerate}
			\item $\lim_{l\to\infty} \overline{b}_l = 0$; 
			\item $\inf_{l} \underline{b}_l > 0$, 
		\end{enumerate}
		we have $\bE[B^\star]=\infty$ under $p=\frac{1}{2}$. 
	\end{theorem}
	
	Note that $B_n/T_{n-1}$ is the bet proportion at $n$-th coup, and general $(-2,1)$-list systems correspond to the case where $\overline{b}_l=1, \underline{b}_l=0$ for any $l$. Theorem \ref{thm.constrained_bet} shows that, if the bet proportion either vanishes or is lower bounded from below as the list length $l$ grows, the largest bet size still has an infinite expectation in a fair game. The following corollary follows from Theorem \ref{thm.constrained_bet}: 
	
	\begin{corollary}\label{cor.p_half}
		For the Labouchere system with any initial list, $\bE[B^\star]=\infty$ if $p=\frac{1}{2}$. 
	\end{corollary}
	
	Combining Corollaries \ref{cor.p_not_half} and \ref{cor.p_half}, we conclude that for the Labouchere system, $\bE[B^\star]=\infty$ if and only if $p\le \frac{1}{2}$, solving the open conjecture in \cite{grimmett2001one}. It also follows directly from Theorems \ref{thm.general_list_system} and \ref{thm.constrained_bet} that for the Fibonacci system, $\bE[B^\star]=\infty$ if and only if $p\le \frac{1}{2}$, recovering the result in \cite{ethier2010doctrine}. Generalizing the arguments to $(-1,0)$-list systems, this also recovers the famous St. Petersburg paradox that $\bE[B^\star]=\infty$ in the martingale system under $p=\frac{1}{2}$. 
	
	Based on Theorem \ref{thm.constrained_bet}, a natural question would be that whether $\bE[B^\star]=\infty$ holds in any $(-2,1)$-list systems. We have the following partial result: 
	\begin{theorem}\label{thm.moment_bounds}
		For any $(-2,1)$-list system and  $\epsilon>0$, the following holds under $p=\frac{1}{2}$:
		\begin{align*}
		\bE\left[B^\star(1\vee \log B^\star)^{-(1+\epsilon)}\right] < \infty, \qquad \bE[B^\star (1\vee \log B^\star)] = \infty. 
		\end{align*}
	\end{theorem}
	Theorem \ref{thm.moment_bounds} shows that, the moment $\bE[(B^\star)^\alpha]$ always has a phase transition at $\alpha=1$ in a fair game. However, the exact answer for $\alpha=1$ is still unknown, and we leave it as a conjecture: 
	\begin{conjecture}
		For any $(-2,1)$-list systems, $\bE[B^\star]=\infty$ under $p=\frac{1}{2}$. 
	\end{conjecture}
	
	\section{Proof of Theorems \ref{thm.general_list_system} and \ref{thm.moment_bounds}}
	In this section, we first prove Theorem \ref{thm.moment_bounds}, and then apply Theorem \ref{thm.moment_bounds} to proving Theorem \ref{thm.general_list_system}. 
	\subsection{Proof of Theorem \ref{thm.moment_bounds}}
	
	We make use of the asymptotic tail behavior of the stopping time $N$ in the $(-2,1)$-list system. 
	\begin{lemma}\cite{ethier2008absorption}\label{lemma.N_asymptotics}
		For $p > \frac 13$, we have
		\begin{align*}
		\bP_{l_0} (N \geq n+1) \sim D_{l_0} (n) n^{-\frac 32} \kappa^\frac n3
		\end{align*}
		where $l_0$ is the length of the initial list, $D_{l_0} (n)$ is a constant only depending  on $l_0$ and $n\pmod 3$, and $\kappa \triangleq \frac{27} {4} p (1-p)^2 <1 $.
	\end{lemma}
	Based on Lemma \ref{lemma.N_asymptotics}, we are about to prove Theorem \ref{thm.moment_bounds}. We first show that $\bE\left[B^\star(1\vee\log B^\star)^{-(1+\epsilon)}\right] < \infty$. Under $p=\frac{1}{2}$, the target sequence $\{T_n\}$ is a martingale, with $\bE[T_n]=T_0$. By Doob's maximal inequality, for any $\lambda>0$, 
	\begin{align*}
	\bP\left(\max_{0\le m\le n} T_m \ge \lambda \right) \le \frac{\bE[T_n]}{\lambda} = \frac{T_0}{\lambda}. 
	\end{align*}
	Note that $B_n\le T_{n-1}$, for $\lambda\ge 2$ we therefore have
	\begin{align*}
	\bP\left(\max_{1\le m\le n} B_m(1\vee\log B_m)^{-(1+\epsilon)} \ge \lambda \right) &=
	\bP\left(\max_{1\le m\le n} B_m \ge C\lambda(\log \lambda)^{1+\epsilon} \right) \\
	&\le \bP\left(\max_{0\le m\le n-1} T_m \ge C\lambda(\log \lambda)^{1+\epsilon} \right) \\ 
	&\le \frac{T_0}{C\lambda (\log \lambda)^{1+\epsilon}}
	\end{align*}
	where $C>0$ is some universal constant. As a result,
	\begin{align*}
	\bE\left[\max_{1\le m\le n} B_m(1\vee\log B_m)^{-(1+\epsilon)} \right] &= \int_0^\infty \bP\left(\max_{1\le m\le n} B_m(1\vee\log B_m)^{-(1+\epsilon)} \ge \lambda \right) d\lambda \\
	&\le 2 +  \int_2^\infty \frac{T_0}{C\lambda (\log \lambda)^{1+\epsilon}} d\lambda < \infty
	\end{align*}
	where in the last step we have used that $$\int_2^\infty \frac{dx}{x(\log x)^{1+\epsilon}}<\infty.$$ 
	Choosing $n\to\infty$, by monotone convergence we arrive at $\bE\left[B^\star(1\vee\log B^\star)^{-(1+\epsilon)}\right] < \infty$. 
	
	Now we show that $\bE[B^\star(1\vee \log B^\star)] = \infty$. We recall the following Fenchel--Young inequality: 
	\begin{align*}
	xy \le \psi(x) + \psi^\star(y)
	\end{align*}
	where $\psi(\cdot)$ is convex, and $\psi^\star(y)=\sup_{x}(xy-\psi(x))$ is the Fenchel--Legendre dual of $\psi$. For $\psi(x) = e^{cx}$ with $c>0$, we have 
	\begin{align*}
	\psi^\star(y) = \sup_{x\in \bR} \left(xy - e^{cx} \right)= \frac{y}{c}\left(\log \frac{y}{c} - 1\right), 
	\end{align*}
	and therefore
	\begin{align*}
	\bE[NB^\star] \le \bE[\psi(N)] + \bE[\psi^\star(B^\star)] = \bE[e^{cN}] + \frac{1}{c}\bE\left[B^\star \left(\log \frac{B^\star}{c} - 1\right) \right]. 
	\end{align*}
	By Lemma \ref{lemma.N_asymptotics}, for $c>0$ sufficiently small we have $\bE[e^{cN}]<\infty$. Moreover, \cite{grimmett2001one} shows that
	\begin{align*}
	\bE[NB^\star] \ge \bE\left[\sum_{n=1}^N B_n\right] = \infty. 
	\end{align*}
	A combination of the previous two inequalities yields $\bE[B^\star(1\vee \log B^\star)]=\infty$. 
	
	\subsection{Proof of Theorem \ref{thm.general_list_system} and Corollary \ref{cor.p_not_half}}
	Now we prove Theorem \ref{thm.general_list_system} using Theorem \ref{thm.moment_bounds} and change of measure. 
	
	Fix any $p>\frac{1}{2}$, let $P$ be the probability measure over the betting process under winning probability $p$, and $Q$ be the counterpart under winning probability $\frac{1}{2}$. Note that for any sample path $\omega$ with stopping time $N=n$, there must be $\frac{n}{3}+c$ wins and $\frac{2n}{3}-c$ losses, where $c$ is a constant depending only on the initial length $l_0$ and $n\pmod 3$. As a result, the likelihood ratio is
	\begin{align*}
	\frac{dP}{dQ}(\omega) = \frac{p^{\frac{n}{3}+c} (1-p)^{\frac{2n}{3}-c}}{2^{-n}} = \left(\frac{p}{1-p}\right)^c\cdot \left(\frac{p(1-p)^2}{\frac{1}{2}(1-\frac{1}{2})^2}\right)^{\frac{n}{3}} \le C\rho^n
	\end{align*}
	where $C>0, \rho\in (0,1)$ are numerical constants independent of $n$, and we have used that the function $p\mapsto p(1-p)^2$ is strictly decreasing in $p\in [\frac{1}{3},1]$. As a result, 
	\begin{align*}
	\bE_P[B^\star] = \bE_Q \left[ B^\star\cdot \frac{dP}{dQ} \right] \le C\cdot \bE_Q [\rho^N B^\star]. 
	\end{align*}	
	
	Since $T_n\le T_{n-1}+B_n \le 2T_{n-1}$ in any list system, $B^\star\le \max_{0\le n\le N}T_n \le T_0\cdot 2^N$, and therefore 
	\begin{align*}
	\bE_Q [\rho^N B^\star] \le T_0^{\epsilon}\cdot \bE_Q [(\rho 2^\epsilon)^N (B^\star)^{1-\epsilon}]
	\end{align*}
	for any $\epsilon>0$. Choosing $\epsilon>0$ small enough such that $\rho 2^\epsilon<1$, by Theorem \ref{thm.moment_bounds} we conclude that $\bE_P[B^\star]<\infty$. 
	
	For $p\in (\frac{1}{3},\frac{1}{2})$, we use the same argument to obtain $\frac{dP}{dQ} \ge C\rho^N$ for some $\rho>1$. Then
	\begin{align*}
	\bE_P[B^\star] \ge C\cdot \bE_Q [\rho^N B^\star] \ge CT_0^{-\epsilon} \cdot \bE_Q [(\rho 2^{-\epsilon})^N (B^\star)^{1+\epsilon}], 
	\end{align*}
	and by choosing $\epsilon>0$ small enough, Theorem \ref{thm.moment_bounds} yields $\bE_P[B^\star] = \infty$. 
	
	Finally, for $p\le \frac{1}{3}$, we have $\bE[\sup_{0\le n< N}l_n]=\infty$ by the theory of asymmetric random walk. Hence, by assumption we have
	\begin{align*}
	\bE[B^\star] \ge c_1\bE[\sup_{0\le n<N}l_n] + c_2 = \infty
	\end{align*}
	as desired. The proof of Theorem \ref{thm.general_list_system} is completed. 
	
	As for Corollary \ref{cor.p_not_half}, it suffices to verify that the condition $B_{n}\ge c_1l_{n-1}+c_2$ holds for the Labouchere system. Let $a>0$ be the minimum number in the initial list $L_0$, a simple induction on $n$ yields that $B_n\ge a(l_{n-1}-l_0)_+$, which shows that the condition is fulfilled with $c_1=a>0, c_2=-al_0$. 
	
	\section{Proof of Theorem \ref{thm.constrained_bet} and Corollary \ref{cor.p_half}}
	In this section, we first use a recursive representation of the optimal list system to prove Theorem \ref{thm.constrained_bet}. Then we investigate the specific properties of the Labouchere system and show that the condition in Theorem \ref{thm.constrained_bet} holds, thereby proving Corollary \ref{cor.p_half}. 
	
	\subsection{Proof of Theorem \ref{thm.constrained_bet}}
	If $\inf_l \underline{b}_l\ge c>0$, we have $B^\star \ge c\max_{0\le n\le N} T_n$, which has an infinite expectation \cite{grimmett2001one}. Now we assume that $\lim_{l\to\infty} \overline{b}_l=0$ and prove Theorem \ref{thm.constrained_bet} by contradiction. We first introduce the following definition: 
	\begin{definition}\label{def.optimal_strategy}
		For any $x>0$ and $l\in \{1,2,\cdots\}$, we define $f(x,l)$ to be the infimum of $\bE[B^\star]$ over all possible $(-2,1)$-list systems with initial target $x$ and initial length $l$, such that $B_n\le \overline{b}_{l_{n-1}}T_{n-1}$ for any $n$. 
	\end{definition}
	
	Definition \ref{def.optimal_strategy} considers an \emph{optimal} $(-2,1)$-list system with initial target $x$ and initial length $l$, where optimality is measured in terms of a smallest expectation of the largest bet size $B^\star$. The quantity $f(x,l)\in \bR_+\cup \{+\infty\}$ is the corresponding expectation, and it is well-defined even if the optimal list system does not exist. The next lemma presents recursive relations between $f(x,l)$ with different $l$. 
	
	\begin{lemma}\label{lemma.basic_f}
		There exists some sequence $\{a_l\}$ taking value in $\bR_+\cup\{+\infty\}$ such that $f(x,l)=xa_l$ for any $x>0$. Moreover, the sequence $\{a_l\}$ satisfies the following inequalities: 
		\begin{align*}
		a_l &\ge \min_{b\in [0,\overline{b}_l]} \frac{\max\{b, (1-b)a_{l-2}\} + \max\{b, (1+b)a_{l+1}\} }{2}, \qquad l\ge 3 \\
		a_1 &\ge a_2 + \frac{1}{2} \ge a_3 + 1. 
		\end{align*}
	\end{lemma}
	\begin{proof}
		When the initial target $x$ is scaled by $\lambda>0$, we may always scale all bet sizes by $\lambda$ to arrive at a new list system with the initial target $\lambda x$, and vice versa. Hence, $f(x,l)$ is proportional to $x$, and $f(x,l)=xa_l$. 
		
		For $l\ge 3$ and any $(-2,1)$-list system, let $b\in [0,\overline{b}_l]$ be any bet size at the first coup with initial target $T_0=1$ and initial length $l$. Let $B_1^\star, B_2^\star$ be the largest bet sizes (excluding the first bet) after winning/losing the first coup, respectively. Then by definition of $f(x,l)$, we have
		\begin{align*}
		\bE B_1^\star &\ge f(1-b, l-2) = (1-b)a_{l-2}, \\
		\bE B_2^\star &\ge f(1+b, l+1) = (1+b)a_{l+1}. 
		\end{align*}
		Note that $B^\star $ is either $\max\{b, B_1^\star\}$ or $\max\{b, B_2^\star\}$, we have
		\begin{align*}
		\bE[B^\star] &= \frac{\bE\max\{b,B_1^\star\} + \bE\max\{b,B_2^\star\}}{2} \\
		&\ge \frac{\max\{b,\bE B_1^\star\} + \max\{b,\bE B_2^\star\}}{2} \\
		&\ge \frac{\max\{b,(1-b)a_{l-2}\} + \max\{b,(1+b)a_{l+1}\}}{2}
		\end{align*}
		where the first inequality is due to the convexity of $x\mapsto \max\{b,x\}$. Note that this inequality holds for any list systems, taking infimum over the LHS gives the desired inequality for $l\ge 3$. The other inequalities for $l\le 2$ can be established analogously. 
	\end{proof}
	
	Based on Lemma \ref{lemma.basic_f}, we may investigate more properties of $a_l$. If $a_1=\infty$, it is obvious that $a_l=\infty$ for any $l\in\naturals$ (since any initial list may evolve into length one with non-zero probability), and Theorem \ref{thm.constrained_bet} holds. Next we show that $a_1<\infty$ is impossible. Assume by contradiction that $a_1<\infty$, we will have the following lemma. 
	\begin{lemma}\label{lemma.decreasing}
		If $a_1<\infty$, the sequence $\{a_l\}$ will be strictly decreasing, i.e., $a_1>a_2>a_3>\cdots$. 
	\end{lemma}
	\begin{proof}
		For $l\ge 3$, by Lemma \ref{lemma.basic_f} we have
		\begin{align*}
		a_l &\ge \min_{b\in [0,\overline{b}_l]}\frac{ (1-b)a_{l-2} + (1+b)a_{l+1}}{2} \\
		&\ge \min_{b\in [0,1]}\frac{ (1-b)a_{l-2} + (1+b)a_{l+1}}{2} \\
		&= \min\left\{\frac{a_{l-2}+a_{l+1}}{2}, a_{l+1}\right \},
		\end{align*}
		where in the last step we have used the fact that an affine function attains its minimum at the boundary. Consequently, if we already know that $a_1\ge a_2\ge \cdots\ge a_l$, we must also have $a_l\ge a_{l+1}$. Hence, by induction on $l$, the sequence $\{a_l\}$ is decreasing. 
		
		To show strict decreasing property, by Lemma \ref{lemma.basic_f} again we have
		\begin{align*}
		a_l &\ge \min_{b\in [0,\overline{b}_l]}\frac{\max\{ b, (1-b)a_{l-2}\} + (1+b)a_{l+1}}{2} \\
		&\ge \min_{b\in [0,1]}\frac{\max\{b, (1-b)a_{l-2}\} + (1+b)a_{l+1}}{2} \\
		&= \frac{1}{2}\min_{b\in [0,1]} \max\{ b+(1+b)a_{l+1}, (1-b)a_{l-2}+(1+b)a_{l+1}\}. 
		\end{align*}
		For real numbers $r_1,r_2,s_1,s_2$ with $r_1>0\ge r_2, s_1\le s_2, r_1+s_1\ge r_2+s_2$, straightforward computation yields
		\begin{align*}
		\min_{x\in [0,1]} \max\{r_1x+s_1, r_2x+s_2\} = \frac{r_1s_2-r_2s_1}{r_1-r_2}. 
		\end{align*}
		Hence, 
		\begin{align*}
		a_l \ge \frac{2a_{l-2}a_{l+1}+a_{l-2}+a_{l+1}}{2(a_{l-2}+1)} = a_{l+1} + \frac{a_{l-2}-a_{l+1}}{2(a_{l-2}+1)}. 
		\end{align*}
		If we have $a_l=a_{l+1}$, we will also have $a_{l-2}=a_{l+1}$ based on the previous inequality. Due to the decreasing property of $\{a_l\}$, $a_{l-1}=a_l$ also holds, and repeating this process yields $a_2=a_3$, a contradiction to Lemma \ref{lemma.basic_f}. Hence $a_l>a_{l+1}$ for any $l$. 
	\end{proof}
	
	Based on Lemmas \ref{lemma.basic_f} and \ref{lemma.decreasing}, we are about to arrive at the desired contradiction. Fix any $\epsilon>0$ such that $\rho\triangleq \frac{1-\epsilon}{1+\epsilon}+(\frac{1-\epsilon}{1+\epsilon})^2>1$. Since $\lim_{l\to\infty}\overline{b}_l = 0$, we take $l_0>0$ large enough such that $\overline{b}_l<\epsilon$ for any $l>l_0$. Then for $l>l_0$, Lemma \ref{lemma.basic_f} yields
	\begin{align*}
	a_l &\ge \min_{b\in [0,\overline{b}_l]}\frac{ (1-b)a_{l-2} + (1+b)a_{l+1}}{2} \\
	&\ge \min_{b\in [0,\epsilon]}\frac{ (1-b)a_{l-2} + (1+b)a_{l+1}}{2} \\
	&= \min_{b\in [0,\epsilon]}\frac{ (a_{l+1}-a_{l-2})b + a_{l+1} + a_{l-2}}{2} \\
	&= \frac{ (a_{l+1}-a_{l-2})\epsilon + a_{l+1} + a_{l-2}}{2}
	\end{align*}
	where in the last step we have used $a_{l+1}\le a_{l-2}$ by Lemma \ref{lemma.decreasing}. A rearrangement of the previous inequality gives
	\begin{align*}
	a_l - a_{l+1} \ge \frac{1-\epsilon}{1+\epsilon}\cdot (a_{l-2} - a_l)
	\end{align*}
	for any $l>l_0$. Similarly, 
	\begin{align*}
	a_{l+1} - a_{l+2} &\ge \frac{1-\epsilon}{1+\epsilon}\cdot (a_{l-1} - a_{l+1})\\
	& \ge \frac{1-\epsilon}{1+\epsilon}\cdot (a_{l} - a_{l+1})\\
	& \ge \left(\frac{1-\epsilon}{1+\epsilon}\right)^2\cdot (a_{l-2}-a_l). 
	\end{align*}
	Adding them together yields
	\begin{align*}
	a_{l} - a_{l+2} \ge \left[\frac{1-\epsilon}{1+\epsilon} + \left(\frac{1-\epsilon}{1+\epsilon}\right)^2 \right]\cdot (a_{l-2}-a_l) = \rho(a_{l-2}-a_l). 
	\end{align*}
	
	Our choice of $\epsilon$ implies $\rho>1$, and therefore $a_{l+2k-2}-a_{l+2k}\ge \rho^k(a_{l-2}-a_l)$ for any $k\in\naturals$ and $l>l_0$. Since $a_{l+2k-2}-a_{l+2k}\le a_1$, and $a_{l-2}>a_l$ by Lemma \ref{lemma.decreasing}, this inequality implies that
	\begin{align*}
	a_1 \ge \rho^k (a_{l-2}-a_l)
	\end{align*}
	for any $k=1,2,\cdots$, a contradiction to the assumption $a_1<\infty$. The proof of Theorem \ref{thm.constrained_bet} is complete. 
	
	\subsection{Proof of Corollary \ref{cor.p_half}}
	First we observe that it suffices to prove the case where the initial list consists of a single positive number. This observation is due to that there is a positive probability to reduce the list length to $l_n=1$ after finitely many coups for any initial list $L_0$. 
	
	To study the combinatorial properties of the Labouchere system, we introduce the following definition: 
	\begin{definition}\label{goodlist}
		A list of positive real numbers $(a_1, a_2, a_3, \cdots, a_n)$ is called \emph{good} if it satisfies the following conditions:
		\begin{itemize}
			\item Every element in the list is positive, i.e., $a_i > 0$ for any $i$; 
			\item The list is non-decreasing, i.e., $a_1 \leq a_2 \leq \cdots \leq a_n$;
			\item The difference of the list is non-decreasing with difference at most $a_1$, i.e., $a_2 - a_1 \leq a_3 -a_2 \leq  \cdots \leq a_n- a_{n-1}\le a_1$. 
		\end{itemize}
	\end{definition}
	
	The key properties of a good list are summarized in the following lemmas.
	\begin{lemma}\label{lemma.goodness}
		If the initial list $L_0$ is good, the list $L_n$ after $n$-th coup is also good for any $n$.
	\end{lemma}
	
	\begin{proof}
		It suffices to prove that, if $L_{n-1}=(a_1,\cdots,a_l)$ is a good list, so is $L_{n}$. Based on the outcome at $n$-th coup, there are only two possibilities: 
		\begin{itemize}
			\item $L_n = (a_1,a_2,\cdots,a_l,a_1+a_l)$, or 
			\item $L_n = (a_2,a_3,\cdots,a_{l-1})$. 
		\end{itemize}
		
		In either case, one can check from Definition \ref{goodlist} directly that $L_{n}$ is a good list, as desired. 
	\end{proof}
	\begin{lemma}\label{lemma.betting_prop}
		If the list $L_{n-1}$ is good and has length $l\ge 2$, in Labouchere system we have
		\begin{align*}
		\frac{B_n}{T_{n-1}} \le \sqrt{\frac{2}{l}} + \frac{2}{l}. 
		\end{align*}
	\end{lemma}
	\begin{proof}
		Let $L_{n-1}=(a_1,\cdots,a_l)$. By definition, for any $k\le l$ we have
		\begin{align*}
		a_l \le a_{l-1} + a_1 \le a_{l-2} + 2a_1 \le \cdots \le a_{l-k} + ka_1.
		\end{align*}
		As a result,
		\begin{align*}
		a_l \le \frac{1}{k}\sum_{j=1}^k (a_{l-j} + ja_1) \le \frac{1}{k}\sum_{j=1}^l a_{j} + \frac{k+1}{2}\cdot a_1. 
		\end{align*}
		
		Note that the current bet size is $B_n=a_1+a_l$, and the current target is $T_{n-1}=\sum_{j=1}^l a_j$. Hence, for any $k\le l$ we have
		\begin{align*}
		\frac{B_n}{T_{n-1}} = \frac{a_1+a_l}{\sum_{j=1}^l a_j} \le \frac{(k+3)a_1}{2\sum_{j=1}^l a_j} + \frac{1}{k} \le \frac{k+3}{2l} + \frac{1}{k}. 
		\end{align*}
		Setting $k=\lceil \sqrt{2l}\rceil\le l$ arrives at
		\begin{align*}
		\frac{B_n}{T_{n-1}} \le \frac{\lceil \sqrt{2l}\rceil + 3}{2l} + \frac{1}{\lceil \sqrt{2l}\rceil} \le \frac{\sqrt{2l}+4}{2l} + \frac{1}{\sqrt{2l}} = \sqrt{\frac{2}{l}} + \frac{2}{l}, 
		\end{align*}
		as claimed. 
	\end{proof}
	
	Note that the initial list $L_0$ consisting of a single positive number is good, by Lemma \ref{lemma.goodness} we know that all future lists $L_n$ are also good. Moreover, by setting
	\begin{align*}
	\overline{b}_l = \min\left\{\sqrt{\frac{2}{l}} + \frac{2}{l}, 1\right\}, 
	\end{align*}
	by Lemma \ref{lemma.betting_prop} we know that $B_n\le \overline{b}_{l_{n-1}}T_{n-1}$ always holds. Note that $\lim_{l\to\infty} \overline{b}_l=0$, Theorem \ref{thm.constrained_bet} yields $\bE[B^\star]=\infty$ in Labouchere system, as desired.

	\section{Acknowledgement}
	The authors would like to thank Stewart Ethier for raising this question and helpful suggestions in improving this paper, and Persi Diaconis for helpful discussions. 
	
	\bibliographystyle{alpha}
	\bibliography{di}

\begin{thebibliography}{Dow80}

\bibitem[Bar87]{barbier1887generalisation}
{\'E}mile Barbier.
\newblock G{\'e}n{\'e}ralisation du probl{\`e}me r{\'e}solu par {M. J.
  Bertrand}.
\newblock {\em Comptes Rendus des S{\'e}ances de l'Acad{\'e}mie des Sciences},
  105:407, 1887.

\bibitem[Ber87]{bertrand1887solution}
Joseph Bertrand.
\newblock Solution d'un probl{\`e}me.
\newblock {\em Comptes Rendus des S{\'e}ances de l'Acad{\'e}mie des Sciences},
  105:369, 1887.

\bibitem[Dow80]{downton1980note}
F~Downton.
\newblock A note on {L}abouchere sequences.
\newblock {\em Journal of the Royal Statistical Society, Series A},
  143:363--366, 1980.

\bibitem[Eth08]{ethier2008absorption}
Stewart~N Ethier.
\newblock Absorption time distribution for an asymmetric random walk.
\newblock In SN~Ethier, J~Feng, and RH~Stockbridge, editors, {\em Markov
  Processes and Related Topics: A Festschrift for Thomas G. Kurtz}, pages
  31--40. IMS Collections 4, Institute of Mathematical Statistics, Beachwood,
  OH, 2008.

\bibitem[Eth10]{ethier2010doctrine}
Stewart~N Ethier.
\newblock {\em {The Doctrine of Chances: Probabilistic Aspects of Gambling}}.
\newblock Springer Science \& Business Media, 2010.

\bibitem[GS01]{grimmett2001one}
Geoffrey Grimmett and David Stirzaker.
\newblock {\em {One Thousand Exercises in Probability}}.
\newblock Oxford University Press, Oxford, 2001.

\end{thebibliography}

\end{document}